\documentclass[11pt,reqno]{amsart}

\usepackage{amsmath, amsfonts, amsthm, amssymb}

\newtheorem{Theorem}{Theorem}[section]
\newtheorem{Lemma}{Lemma}[section]
\newtheorem{Proposition}{Proposition}[section]

\theoremstyle{definition}
\newtheorem{Definition}{Definition}[section]

\theoremstyle{remark}
\newtheorem{Remark}{Remark}[section]

\numberwithin{equation}{section}

\renewcommand{\u}{{\bf u}}

\newcommand{\R}{{\mathbb R}}
\newcommand{\Dv}{{\rm div}}

\def\f{\frac}
\renewcommand{\O}{\Omega}

\def\D{\Delta }

\def\hf1{^\f{1}{1-\xi^2}}

\def\be{\begin{equation}}
\def\en{\end{equation}}
\def\bs{\begin{split}}
\def\es{\end{split}}

\renewcommand{\v}{{\bf v}}

\author{Cheng Yu}
\address{Department of Mathematics, University of Pittsburgh,
                           Pittsburgh, PA 15260.}
\email{chy39@pitt.edu}
\title[Global Solutions to the  Navier-Stokes-Vlasov Equations]
{Global Weak Solutions to the Incompressible Navier-Stokes-Vlasov Equations}


\date{\today}

\begin{document}
\begin{abstract}
In this paper, the system of particles coupled with fluid is considered. The particles are described by a Vlasov equation, and the fluid is governed by the forced incompressible Navier-Stokes equations. The interaction with fluid phase governed by Navier-Stokes equations is taken into account through a source term.
The resulting system, namely Navier-Stokes-Vlasov equations, is shown to have global weak solutions in two and three spatial dimensions, and to have a unique global solution in two spatial dimensions.

\end{abstract}

\maketitle

\section{Introduction}
The fluid-particle flow has received much attention recently. It has a wide range of applications \cite{CP,F,O,RM,S,W} in the modeling of reaction flows of sprays, atmospheric pollution modeling,
chemical engineering or waste water treatment, dust collecting units, and so on. Consequently,
the description of the motion of fluid with charged particles has became a crucial problem. Due to the huge number of particles, it is very hard to calculate each particle's trajectory especially with the addition of the dimension of physical domain. Actually, we are not interested in the precise locations of each particle but the average behavior of collection of particles. For that purpose, we use kinetic method to describe the effect of particles on the fluid. On the other hand, macroscopic models such as hydrodynamic ones are commonly used in fluid. The interaction with fluid phase governed by fluid equations is taken into account through a force term. In this paper, we study a coupled system consisting of the Navier-Stokes equations of fluid dynamics and the
Vlasov equation of the cloud of particles, namely, incompressible Navier-Stokes-Vlasov equations. We refer to \cite{AKS,BDGM,CP,F,O,RM,S,W} for some physical introduction
to Navier-Stokes-Vlasov equations and the related applications.

The objective of this paper is to establish the global weak solutions for the following partial differential equations
\begin{equation}
\label{1.1}
\begin{split}
&\u_t+(\u\cdot\nabla)\u+\nabla p-\mu\D\u=-\int_{\R^d}(\u-\v)f\,d\v,
\\&\Dv\u=0,
\\&f_t+\v\cdot\nabla_xf+\Dv_{\v}((\u-\v)f)=0,
\end{split}
\end{equation}
in $\O\times\R^d\times(0,T),$ where $\O\subset\R^d$ is a bounded domain, $d=2, 3$,
 $\u$ is the velocity of the fluid, $p$ is the pressure, $\mu$ is the kinematic viscosity of the fluid.
Without loss of generality,
we take $\mu=1$ throughout the paper. A distribution $f(t,x,\v)$
depends on the time $t\in[0,T]$, the physical position $x\in\O$ and the velocity of particle $\v\in\R^d$. The number of particles enclosed at $t\geq 0$ and location $x\in\O$ in the volume element $d\v$ is given by $f(t,x,\v)\,d\v$.

The system is completed by the initial data
\begin{equation}
\label{1.2}
\u(0,x)=\u_0(x),\quad f(0,x,\v)=f_0(x,\v),
\end{equation}
and with the following boundary conditions:

\begin{equation}
\label{boundary condition}
\u=0\;\;\text{on}\;\;\partial\O,\;\;\;\text{ and }\;\;\;
f(t,x,\v)=f(t,x,\v^{*})\;\; \text {for } x\in\partial\O,\;\v\cdot\nu(x)<0
\end{equation}
where $\v^{*}=\v-2(\v \cdot\nu(x)) \nu(x)$ is the specular velocity, $\nu(x)$ is the outward normal to $\O$.

When the distribution function $f$ is absent, \eqref{1.1} reduces to the incompressible Navier-Stokes equations. For the incompressible Navier-Stokes equations, we refer to the books \cite{CF,DG,L,T} for fundamental ideas on
 the global weak solutions, stability, and limitations of the current tools. Meanwhile, there has been extensive mathematical study on the Vlasov equation and the related problems, for example, see \cite{CH,JN}.

The mathematical analysis of the coupled systems is far from being complete but recently has received much attention. Early work \cite{H} established the global existence results for Stokes-Vlasov system in a bounded domain. The existence theorem for weak solutions has been extended in \cite{BDGM}, where the author did not neglect the convection term and considered the Navier-Stokes-Vlasov equations within a periodic domain. The global existence of smooth solutions with small data for Navier-Stokes-Vlasov-Fokker-Planck equations was obtained in \cite{GHMZ}.
More Recently, the existence of global weak solutions of Navier-Stokes-Vlasov-Poisson system with corresponding boundary value problem was obtained in \cite{AKS}. Meanwhile, there are considerable works in the direction of hydrodynamic limits, we refer the reader to \cite{CP,GJV,GJV2,MV2}. In works \cite{CP,GJV,GJV2,MV2}, the authors used some scaling issues and convergence methods, such as compactness and relative entropy method, to investigate the hydrodynamic limits. A key idea in \cite{GJV,GJV2} is to have the dissipation of a certain free energy associated to the whole space, and control its dissipation rate. For the compressible version,
local strong solutions of Euler-Vlasov equations was established in \cite{BD}.
Global existence of weak solutions for compressible Navier-Stokes equations coupled to Vlasov-Fokker-Planck equations was established in \cite{MV}.

There is no existence theory  available for the Navier-Stokes-Vlasov equations in a bounded domain. Two different variable sets in the coupling term present new challenges to the existence theory. The other major difficulty arises from the Vlasov equations, such as the compactness properties of the distribution function $f(t,x,\v).$ The objective of this work is to establish the global existence of weak solutions with large data in the three dimensional spaces, and to establish the uniqueness in the two dimensional space.

In this paper we shall study the initial-boundary value problem \eqref{1.1}-\eqref{boundary condition} for Navier-Stokes-Vlasov equations and establish the global existence and uniqueness of weak solutions with large data in certain functional spaces.  Motivated by the works of \cite{DG,M,MV}, we propose a new approximation scheme. The key point of approximation is to control the modified force term of regularized Navier-Stokes equations.
 The existence and uniqueness of the modified Vlasov equation is classically obtained, for example, see \cite{BP,DL,H}.
  The controls of $\int_{\R^d}f d\v$ and $\int_{\R^d}\v f\,d\v$ ensure that the modified force term of the Navier-Stokes equations has enough regularity.  Thus, we can solve the regularized Navier-Stokes equations. The compactness properties of the system will allow us to pass the limit to recover the original system. We shall also establish the uniqueness of the weak solutions in the two dimensional space.

The rest of the paper is organized as follows. In Section 2, we deduce a priori estimates for \eqref{1.1}-\eqref{boundary condition}, give the definition of weak solutions, and state our main results. In Section 3, we construct an approximation scheme to \eqref{1.1}-\eqref{boundary condition}, establish its global existence. In Section 4, we establish the uniqueness for the global weak solutions in the two dimensional space.
\bigskip


\section{A Priori Estimates and Main Results}
Here we define the energy functional of the particles density:
\begin{equation*}
F(f):=\int_{\O}\int_{\R^d}f(1+|\v|^2)\,d\v\, dx.
\end{equation*}
If $\u=\text{constant}$, $F(f)$ is an energy functional to the third equation in \eqref{1.1}. When $\u\neq\text{constant}$, we will have the following energy inequality, more precisely:
\begin{Lemma}
\label{L1} The system \eqref{1.1} has an energy functional:
\begin{equation*}
E(t):=\left(\int_{\O}\frac{1}{2}|\u|^2\,dx+F(f)\right)(t).
\end{equation*}
If $d=2,3$ and $(\u,f)$ is a smooth solution to system \eqref{1.1} such that
\begin{equation}
\label{re1}
\u\in L^{\infty}(0,T;L^2(\O))\cap L^{2}(0,T;H^{1}_0(\O));
\end{equation}
\begin{equation}
\label{re2}
f(1+|\v|^2)\in L^{\infty}(0,T;L^{1}(\O\times\R^d)),
\end{equation}
then, for all $t<T,$ we have:
\begin{equation}
\label{2.1}
\frac{d}{dt}E(t)=-\left(\int_{\O}|\nabla\u|^2\,dx+\int_{\O}\int_{\R^d}f|\u-\v|^2\,d\v\, dx\right)\leq 0.
\end{equation}
\end{Lemma}
\begin{proof}
Multiplying by $\u$ the both sides of the first equation in \eqref{1.1}, and integrating over $\O$ and by parts,
we have
\begin{equation}
\label{2.2}
\frac{d}{dt}\int_{\O}\frac{1}{2}|\u|^2\,dx+\int_{\O}|\nabla\u|^2\,dx=-\int_{\O}\int_{\R^d}f(\u-\v)\cdot\u\, d\v\, dx.
\end{equation}
Multiplying by $(1+\frac{1}{2}|\v|^2)$ the both sides of the third equation in \eqref{1.1},  integrating over $\O$, and using integration by parts, one obtains that
\begin{equation}
\label{2.3}
\frac{d}{dt}F(f)(t)+\int_{\O}\int_{\R^d}f|\u-\v|^2\,d\v \,dx=\int_{\O}\int_{\R^d}f(\u-\v)\cdot\u\, d\v\, dx.
\end{equation}
Using \eqref{2.2}-\eqref{2.3}, one obtains \eqref{2.1}, and \eqref{re1}-\eqref{re2} are the consequences of \eqref{2.1}.

The proof is complete.

\end{proof}
In what follows, we denote $$m_kf=\int_{\R^d}|\v|^kf\, d\v,\;\;\text{ and } \;\;M_kf=\int_{\O}\int_{\R^d}|\v|^kf\,d\v dx.$$
Clearly,
\begin{equation*}
M_kf=\int_{\O}m_kf\,dx
\end{equation*}
Here we state the following lemmas which are due to \cite{H}:
\begin{Lemma}
\label{L2}
Suppose that $(\u,f)$ is a smooth solution to
\eqref{1.1}. If $f_0\in L^{p}$ for any $p>1$, we have
\begin{equation*}
\|f(t,x;\v)\|_{L^p}\leq e^{dT}\|f_0\|_{L^{p}}, \text{ for any } t\geq0;
\end{equation*}
and if $|\v|^kf_0\in L^{1}(\O\times\R^d),$ then we have
\begin{equation*}
\int_{\O\times\R^d}|\v|^kfd\v dx\leq C(d,T)\left(\left(\int_{\O\times\R^d}|\v|^kf_0\,d\v dx\right)^{\frac{1}{d+k}}+(||f_0||_{L^{\infty}}+1)\|\u\|_{L^{r}(0,T;L^{d+k})}\right)^{d+k}
\end{equation*}
for all $o\leq t\leq T$ where the constant $C(d,T)>0$ depends only on $d$ and $T$.
\end{Lemma}

\begin{Lemma}
\label{L3}
Under hypotheses of Lemma \ref{L2} and $d=3$, the density $m_0f$ and the mean velocity $m_1f$ have the following estimates for all $0\leq t\leq T,$
\begin{equation*}
\begin{split}
&\|m_0f\|_{L^2(\O)}\leq C(1+\|f_0\|_{L^{\infty}(0,T;L^{\infty}(\O\times\R^3)})A^3
\\&\|m_1f\|_{L^{3/2}(\O)}\leq C (1+\|f_0\|_{L^{\infty}(0,T;L^{\infty}(\O\times\R^3)})A^{4},
\end{split}
\end{equation*}
where $A=(\int_{\O}\int_{\R^3}|\v|^{3}f_0\,dx d\v)^{\frac{1}{6}}+(\|f_0\|_{L^{\infty}(0,T;L^{\infty}(\O\times\R^3))}+1)\|\u\|_{L^2(0,T;L^{6}(\O))}.$
\end{Lemma}

\begin{Remark}
Similar estimates hold in the two dimensional space.
\end{Remark}

Our main result reads as follows.
\begin{Theorem}
\label{T1} For $d=2,3$ there is a weak solution $(\u,f)$ to the system \eqref{1.1} with the initial data \eqref{1.2} and boundary condition \eqref{boundary condition} for any $T>0.$
\end{Theorem}
The weak solution to system \eqref{1.1}-\eqref{boundary condition} is defined as follows.
\begin{Definition}
A pair $(\u,f)$ is called a weak solution to the system \eqref{1.1}-\eqref{boundary condition} in the sense of distribution if
\begin{itemize}
\item  $\u\in L^{\infty}(0,T;L^{2})\cap L^{2}(0,T;H_{0}^{1})$;
\item $f(t,x,\v)\geq 0, \text{ for any } (t,x,\v)\in (0,T)\times\O\times\R^d$;
\item $f\in L^{\infty}(0,T;L^{\infty}(\O\times\R^d)\cap L^{1}(\O\times\R^d))$;
\item $f|\v|^2\in L^{\infty}(0,T;L^{1}(\O\times\R^d))$;
\item for all $\varphi \in C^{\infty}([0,\infty)\times\O)$ with $\Dv\varphi=0$ we have
\begin{equation*}
\begin{split}
\int_0^{\infty}\int_{\O}(-\u\varphi_t&+\u\cdot\nabla\u\varphi+\nabla\u\cdot\nabla\varphi)\,dx\,dt\\
&=-\int_0^{\infty}\int_{\O\times\R^d}f(\u-\v)\cdot\varphi \,dx\,d\v\, ds+\int_{\O}\u_0\cdot\varphi(0,x)\,dx;
\end{split}
\end{equation*}

\item for all $\phi \in C^{\infty}([0,\infty)\times\O\times\R^d)$ with compact support in $\v$, such that $\phi(T,x,\v)=0,$ we have
\begin{equation*}
\begin{split}
&-\int_0^{T}\int_{\O\times\R^d}f(\phi_t+\v\cdot\nabla_x\phi+(\u-\v)\cdot\nabla_\v\phi)\,dx\,d\v\, dt
\\&=\int_{\O\times\R^d}f_0\phi(0,x,\v)\,dx\, d\v.
\end{split}
\end{equation*}
\item the energy inequality
\begin{equation*}
\begin{split}
&\int_{\O}\frac{1}{2}|\u|^2\;dx+\int_{\O}\int_{\R^d}f(1+|\v|^2)\;dx d\v
\\&+\int_{0}^T\int_{\O}|\nabla\u|^2\,dxdt+\int_{0}^{T}\int_{\O\times\R^d}f|\u-\v|^2\,d\v dx dt
\\&\leq \int_{\O}\frac{1}{2}|\u_0|^2\;dx+\int_{\O}\int_{\R^d}f_0(1+|\v|^2)\;dx
\end{split}
\end{equation*}
holds for $t\in[0,T]$ a.e.
 \end{itemize}
\end{Definition}
In the two dimensional space, we can obtain more regularity and the uniqueness of global weak solution. More precisely, we have:
\begin{Theorem}\label{T2}
 If $\u_0\in H^{1}(\O)$, $f_0\in L^{\infty}(\R^2)\cap L^1(\R^2)$, and $ \int_{\R^2}|\v|^6f_0\,d\v<\infty$, there exists a unique global solution $(\u,f)$ to the system \eqref{1.1} with the initial data \eqref{1.2} and boundary condition \eqref{boundary condition}, such that
 \begin{equation*}
 \begin{split}
&\quad\quad\quad \u \in L^{2}(0,T;H^2_0(\O))\cap C([0,T];H^1_0(\O)),\quad \frac{\partial \u}{\partial t}\in L^{2}((0,T)\times\O),
 \\& \;\;\; \;\;\;\;\;\;\;f\in C([0,T];L^{\infty}(\O\times\R^2))
 \end{split}
 \end{equation*}
 for any $T>0.$
 \end{Theorem}


\bigskip

\section{The Existence of Weak Solutions}
The goal of this section is to show the existence of global weak solutions to \eqref{1.1} with initial data \eqref{1.2} and boundary condition \eqref{boundary condition}. The key idea is to construct an approximation scheme, establish its existence for the global weak solutions, and pass to the limit for recovering the original system. In this section, we shall prove our main result Theorem \ref{T1} in the case $d=3.$ All arguments do work in the case $d=2.$
\subsection{Approximation Scheme}
We regularize the equations \eqref{1.1} and construct a solution of the regularized system of equations. We view the first two equations in \eqref{1.1} as Navier-Stokes equations with a source term $\int_{\R^3}(\u-\v)f\, d\v$. The key idea is to control $\int_{\R^3}(\u-\v)f \,d\v$ in $L^{2}((0,T)\times\O)$ so that we can solve the Navier-Stokes equations directly. For that purpose, we follow the spirit in \cite{MV} to modify the Vlasov equation by truncating the velocity field $\u$: we consider
\begin{equation}
\label{3.0+}
\partial_t f+\v\cdot\nabla_x f+\Dv_{\v}\left((\chi_{\lambda}(\u)-\v)f\right)=0
\end{equation}
where $$\chi_{\lambda}(\u)=\u 1_{\{|\u|\leq\lambda\}}.$$
To preserve the similar energy inequality, we need to modify Navier-Stokes equations accordingly. This can be done by substituting the right hand side of the first equation in \eqref{1.1} by $$-\int_{\R^3}(\u-\v)f\,d\v 1_{\{|\u|\leq\lambda\}}.$$

To establish the global weak solutions, we find a modified Galerkin method particularly convenient. We define the space $H$ as the closure of the space $C_0^{\infty}(\O,\R^3)\cap\{\u:\Dv\u=0\}$ in $L^{2}(\O,\R^3)$. We let $\{\phi_i\}_{i=1}^{\infty}$ be an orthogonal basis of the functional space $H$ and such that
\begin{equation*}
\begin{split}
&\D\phi_i+\nabla P_i=-e_{i}\phi_{i}\;\;\text{ in } \O,
\\& \;\;\;\;\phi_i=0\;\; \text{ on } \partial\O
\end{split}
\end{equation*}
for $i=1,2,3...$. Here $0\leq e_1\leq e_2\leq e_3\leq...\leq e_n\leq ...$ with $e_n\to \infty$ as $n\to\infty.$
Define $ P_m:H\mapsto H_m=\text{span}\{\phi_1,\phi_2,...,\phi_m\}$ as the orthogonal projection.
We propose the following regularized Navier-Stokes equations
 \begin{equation*}
 \begin{split}
 &\partial_t\u_m=P_m\left(\D\u_m-\u_m\cdot\nabla\u_m-\int_{\R^3}(\tilde{\u}-\v)f_m \,d\v1_{\{|\tilde{\u}|\leq\lambda\}}\right),
 \\& \u_m(x,t)\in H_m,\;\;\text{ and }\Dv\u_m=0.
 \end{split}
 \end{equation*}
 Thus the approximation scheme for the Navier-Stokes-Vlasov equations \eqref{1.1} is given by the following system

\begin{equation}
\begin{split}
\label{3.1}
 &\partial_t\u_m=P_m\left(\D\u_m-\u_m\cdot\nabla\u_m-G_m\right),
 \\& \u_m(x,t)\in H_m,\;\;\text{ and }\;\;\Dv\u_m=0.
\\& \partial_tf_m+\v\cdot\nabla_xf_m+\Dv_{\v}(\chi_{\lambda}(\tilde{\u})-\v)f_m)=0,
\end{split}
\end{equation}
with the initial data,
$$\u_m|_{t=0}=P_m\u_0, \quad\quad\quad f_m|_{t=0}=f_0,$$
where $$G_m=\int_{\R^3}(\tilde{\u}-\v)f_m \,d\v1_{\{|\tilde{\u}|\leq\lambda\}}$$ and $$\tilde{\u}\;\text{ is given in } L^{2}(0,T;L^2(\O)).$$
The rest of this subsection is devoted to show the global existence to \eqref{3.1} with its initial data. For given $\tilde{\u}$, we can get enough regularity of $-\int_{\R^3}(\tilde{\u}-\v)f \,d\v 1_{\{|\tilde{\u}|\leq\lambda\}}$ to solve the modified Navier-Stokes equations.

Indeed, we have $$\chi_{\lambda}(\tilde{\u})\in L^{\infty}((0,T)\times\O),$$
due to $$\tilde{\u}\in L^{2}(0,T;L^2(\O)).$$
Considering the following equation
\begin{equation*}
\begin{split}
&\partial_t f+\v\cdot\nabla_xf+\Dv_{\v}((\chi_{\lambda}(\tilde{\u})-\v)f)=0;
\\& f(x,\v,0)=f_0(x,\v),\;\;\;f(t,x,\v)=f(t,x,\v^{*})\;\;\, \text {for } x\in\partial\O,\;\v\cdot\nu(x)<0,
\end{split}
\end{equation*}
where $\v^*=\v-2(\v\cdot\nu(x))\nu(x),$ the existence and uniqueness of the solution can be obtained as in  \cite{BP,DL,H}.

Applying the maximal principle to the above equation, we have
\begin{equation}
\label{3.3}
\|f(t,x,\v)\|_{L^p}\leq C(T)\|f_0\|_{L^{p}},\text{  for any } p>1.
\end{equation}
Thanks to Lemma \ref{L2}, $$\chi_{\lambda}(\tilde{\u})\in L^{\infty}((0,T)\times\O),$$
and $$\int_{\O}\int_{\R^3}|\v|^5f_0\;d\v dx<+\infty,$$
 we have
\begin{equation*}
\int_{\O}\int_{\R^3}|\v|^5fd\v dx<+\infty.
\end{equation*}
This, together with Lemma 1 in \cite{BDGM}, yields
\begin{equation}
\label{3.1+}
\int_{\R^3}f\,d\v \in L^{2}(0,T;L^2(\O)),\text{ and } \int_{\R^3}\v f\,d\v \in L^{2}(0,T;L^2(\O)).
\end{equation}
By \eqref{3.1+}, we get, for all $t>0$, that
$G_m\in L^{2}(0,T;L^2(\O)).$

For each $m$ we define
an approximate solution $\u_m$
of \eqref{3.1} as follows:
\begin{equation*}
\u_m=\sum_{i=1}^{m}g_{i m}(t)\phi_i(x),
\end{equation*}
and hence \eqref{3.1} is equivalent to

\begin{equation}
\label{approximation}
\frac{d g_{m}^{i}(t)}{dt}=-e_{i}g_{m}^{i}(t)-\int_{\O}(\phi_j(x)\cdot\nabla\phi_k(x))\phi_i(x))\,dx\, g_{m}^{j}(t)g_{m}^{k}(t)+\int_{\O}G_m\phi_{i}(x)dx.
\end{equation}
The initial data becomes
\begin{equation*}
\sum_{i=1}^{m}g_{m}^{i}(0)\phi_{i}(x)=P_m\u_0(x),
\end{equation*}
which is equivalent to saying that
\begin{equation}
\label{initialofapproximation}
g_{m}^{i}(0)=(\u_0,\phi_i)\;\;\text{ for } i=1,2,3,...,m.
\end{equation}

So the system \eqref{approximation} with its initial data \eqref{initialofapproximation} can be viewed as an ordinary differential equations in $L^2$ verifying the conditions of the Cauchy-Lipschitz theorem.
Thus it admits a unique maximal solution $$\u_m\in C^{1}([0,T_m];L^{2}(\O)).$$

It is easy to find the energy inequality to regularize Navier-Stokes equations as follows
\begin{equation*}
\int_{\O}|\u_m|^2\,dx+2\int_0^t\int_{\O}|\nabla\u_m|^2\,dxdt\leq\int_{\O}|\u_{m0}|^2\,dx+\int_0^t\int_{\O}G_m\u_m\,dxdt,
\end{equation*}
which, together with $G_m\in\L^{2}(0,T;L^2(\O))$, allows us to take $T_m=T.$

We define an operator
\begin{equation*}
\begin{split}
S:\quad L^{2}((0,T)\times\O)&\mapsto L^{2}((0,T)\times\O)
\\& \tilde{\u}\mapsto \u_m.
\end{split}
\end{equation*}
Here we need to rely on the following lemma:
\begin{Lemma}
\label{L3.2}
The operator $S$ has a fixed point in $L^{2}((0,T)\times\O)$, that is, there is a point $\u_m\in L^{2}((0,T)\times\O)$ such that $S\u_m=\u_m.$
\end{Lemma}
\begin{proof}
Multiplying by $\u_m$ the both sides of \eqref{3.1}, and using integration by parts, one obtains that
\begin{equation}
\begin{split}
\label{energyofapproximation}
&\frac{d}{dt}\int_{\O}\frac{1}{2}|\u_m|^2\,dx
+\int_{\O}|\nabla\u_m|^2\,dx\leq \int_{\O}\left(\int_{\R^3}(\tilde{\u}-\v)f_m\,d\v 1_{\{|\tilde{\u}|\leq\lambda\}}\right)\u_m \,dx.
\end{split}
\end{equation} Considering the force term of the modified Navier-Stokes equations, we have
\begin{equation*}
|\int_{\O\times\R^3}(\tilde{\u}-\v)f_m1_{\{|\tilde{\u}|\leq \lambda\}}\u_m \,d\v\, dx|\leq(\|\int_{\R^3}\v f\,d\v\|_{L^2}+\lambda\|\int_{\R^3} f\,d\v\|_{L^2})^2+\|\u_m\|_{L^2}^2.
\end{equation*}
This, together with \eqref{energyofapproximation}, implies that
\begin{equation*}
\partial_t\int_{\O}\frac{1}{2}|\u_m|^2\,dx+\int_{\O}|\nabla\u_m|^2\,dx\leq \int_{\R^3}|\u_m|^2dx+C(m).
\end{equation*}
By Gronwall's inequality, we have
\begin{equation*}
\sup_{t\in(0,T)}\int_{\O}|\u_m|^2\,dx\leq C(m),
\end{equation*}
which means that
\begin{equation}
\label{w1}
\|S\tilde{\u}_m\|_{L^{2}((0,T);H^{1}_0(\O))}\leq C(m).
\end{equation}

By the first equation in \eqref{3.1},
one obtains that
\begin{equation}
\label{w2}
\|\partial_t S\tilde{\u}_m\|_{L^{2}(0,T;H^{-1}_{0}(\O))}\leq C(m).
\end{equation}
By \eqref{w1} and \eqref{w2}, we conclude that the operator $S$ is compact in $L^{2}(0,T;L^2(\O))$ and the image of the operator $S$ is bounded in $L^{2}(0,T;\O)$. So Schauder's fixed point theorem will give us that the operator $S$ has a fixed point $\u_m$ in $L^{2}(0,T;\O)$.
\end{proof}
Applying Lemma \ref{L3.2}, for any $T>0$, there exists a solution $(\u_m,f_m)$ to the following system
\begin{equation}
\begin{split}
\label{approximation2}
 &\partial_t\u_m=P_m\left(\D\u_m-\u_m\cdot\nabla\u_m-\int_{\R^3}(\u_m-\v)f_m\,d\v1_{|\u_m|\leq \lambda}\right),
 \\& \u_m(x,t)\in H_m,\;\;\text{ and }\;\;\Dv\u_m=0.
\\& \partial_tf_m+\v\cdot\nabla_xf_m+\Dv_{\v}(\chi_{\lambda}(\u_m)-\v)f_m)=0,
\end{split}
\end{equation}
with its initial data $$\u_m(0,x)=P_m\u_0,\;\;\,f(0,x,\v)=f_0(x,\v),$$
and boundary conditions $$\u_m|_{\partial\O}=0, \;\;\text{ and }\;\;\, f(t,x,\v)=f(t,x,\v^*) \;\;\text{ for any}\,\,\, x\in\partial\O,\;\v\cdot\nu(x)<0.$$

Concerning the system \eqref{approximation2} with the initial-boundary data,
we have the following energy inequality
\begin{equation}
\label{3.5}
\begin{split}
&\int_{\O}\frac{1}{2}|\u_m|^2\,dx+\int_{\O\times\R^3}f_m(1+|\v|^2)\;d\v dx\\&+\int_0^T\int_{\O}|\nabla\u_m|^2\,dxdt
+\int_0^T\int_{\O\times\R^3}f_m|\chi_{\lambda}(\u_m)-\v|^2\,d\v dxdt
\\&\leq\int_{\O}\frac{1}{2}|\u_0|^2\,dx+\int_{\O\times\R^3}f_0(1+|\v|^2)\,d\v\,dx,
\end{split}
\end{equation}
due to the fact
\begin{equation*}
\int_{\O\times\R^3}(\u_m-\v)f_m1_{\{|\u_m|\leq\lambda\}}\u_m \,d\v \,dx=\int_{\O\times\R^3}\chi_{\lambda}(\u_m)(\chi_{\lambda}(\u_m)-\v)f_m\,dx\, d\v.
\end{equation*}

Then we have the following result:
\begin{Proposition}
\label{P2}
For any $T>0$, there is a weak solution $(\u^m,f^m)$ to the following system
\begin{equation*}
\begin{split}
 &\partial_t\u_m=P_m\left(\D\u_m-\u_m\cdot\nabla\u_m-\int_{\R^3}(\u_m-\v)f_m\,d\v1_{|\u_m|\leq \lambda}\right),
 \\& \u_m(x,t)\in H_m,\;\;\text{ and }\;\;\Dv\u_m=0.
\\& \partial_tf_m+\v\cdot\nabla_xf_m+\Dv_{\v}(\chi_{\lambda}(\u_m)-\v)f_m)=0,
\end{split}
\end{equation*}
with its initial data $$\u_m(0,x)=P_m\u_0,\;\;\,f(0,x,\v)=f_0(x,\v),$$
and boundary conditions $$\u_m|_{\partial\O}=0, \;\;\text{ and }\;\;\, f(t,x,\v)=f(t,x,\v^*) \;\;\text{ for any}\,\,\, x\in\partial\O,\;\v\cdot\nu(x)<0.$$
In additional, the solution satisfies the following energy inequality:
\begin{equation*}
\begin{split}
&\int_{\O}\frac{1}{2}|\u_m|^2\,dx+\int_{\O\times\R^3}f_m(1+|\v|^2)\;d\v dx\\&+\int_0^T\int_{\O}|\nabla\u_m|^2\,dxdt
+\int_0^T\int_{\O\times\R^3}f_m|\chi_{\lambda}(\u_m)-\v|^2\,d\v dxdt
\\&\leq\int_{\O}\frac{1}{2}|\u_0|^2\,dx+\int_{\O\times\R^3}f_0(1+|\v|^2)\,d\v\,dx,
\end{split}
\end{equation*}
\end{Proposition}
\subsection{Passing to the Limit as $m\to\infty$}
In this section, we will pass the limit as $m$ goes to infinity in the family of approximate
solutions $(\u^{m},f^{m})$  obtained in Proposition \ref{P2}. The estimates in Proposition \ref{P2}
are independent of $m, \lambda$, and those estimates still hold for any $m$.
By \eqref{3.3}, we have
\begin{equation*}
\|f^{m}\|_{L^{\infty}(0,T;L^{p}(\O\times\R^3))}\leq C
\end{equation*}
for all $1\leq p\leq \infty.$
Proposition \ref{P2} yields the following estimates
\begin{equation*}
\begin{split}
&\|\u^{m}\|_{L^{\infty}(0,T;L^{2}(\O))}\leq C,
\\&\|\nabla\u^{m}\|_{L^{2}(0,T;L^{2}(\O))} \leq C.
\end{split}
\end{equation*}

 From the above a priori estimates, we conclude that there exists a function $f$ such that
 $$f^{m}\rightharpoonup f\;\;\text{ weak star in } L^{\infty}(0,T;L^{p}(\O\times\R^3))$$
 for all $p\in(1,\infty).$

 This weak convergence cannot provide us enough information for passing the limit. For our purpose,
 we rely on the following average compactness results for the Vlasov equation due to DiPerna-Lions-Meyer \cite{DLM}:
\begin{Lemma}
\label{averagelemma}
\begin{equation*}
\frac{\partial f^{n}}{\partial t}+\v\cdot\nabla_{x}f^n=\Dv_{\v}(F^n f^n)\;\;\text{ in } \mathcal{D}'(\R^3_x+\R^3_\v\times(0,\infty))
\end{equation*}
where $f^n$ is bounded in $L^{\infty}(0,\infty;L^2_{x,\v}\bigcap L^1_{x,\v}(1+|\v|^2))$, $\frac{F^n}{1+|\v|}$ is bounded in $L^{\infty}((0,\infty)\times \R^3_{\v};L^{2}(\R^3_{x})).$
Then $\int_{\R^3}f^n\eta(\v)\,d\v$ is relatively compact in $L^{q}(0,T;L^{p}(B_{R}))$ for all $R,T<\infty$, $1\leq q <\infty, 1\leq p< 2$ and for $\eta$ such that $\frac{\eta}{(1+|\v|)^\sigma}\in L^1+L^{\infty}, \sigma \in [0,2).$
\end{Lemma}
\begin{Remark}
It is crucial to use this lemma to get the strong convergence of $m_0f^n$ and $m_1f^n$.
\end{Remark}

 Let $m_0f$ and $m_1f$ be the density and mean velocity associate with $f$. Applying Lemma \ref{averagelemma} to the Vlasov equation of \eqref{3.1}, one obtains that


 \begin{equation}
 \label{averageestimate2}
 m_0f^{m}(t,x)\to m_0f(t,x),\;\;\;m_1f^m(t,x)\to m_1f(t,x)
 \end{equation}
 in $L^{q}(0,T;L^{p}(B_R))$ for any positive number $R$ and $1\leq q <\infty, 1\leq p< 2$.

Noticing that the right side of Navier-Stokes equations \eqref{3.1}
$$\int_{\O}(\u^m-\v)f^m\,d\v 1_{\{|\u^m|\leq\lambda\}}$$
is bounded in $L^{\infty}(0,T;L^2(\O))$ when $\lambda$ is fixed,
one obtains that
\begin{equation}
\label{3.6}
\|\partial_t\u^{n}\|_{L^{2}(0,T;H^{-1})}\leq C<\infty.
\end{equation}
By \eqref{3.5}-\eqref{3.6}, applying the Aubin-Lions Lemma, (see \cite{T}), there exist a $\u\in L^{\infty}(0,T;L^2)\cap L^{2}(0,T;H^1_0)$, such that
\begin{equation}
\label{3.2+}
\begin{split}
&\u^{m}\rightharpoonup\u \text{  weak star in } L^{\infty}(0,T;L^2)\;\; \text{ and }\;\; \u^{m}\rightharpoonup \u \text{ in }L^{2}(0,T;H^1_0)
\\&\u^{m}\to\u \text{ strongly in } L^2(0,T;H^{1}_{0}).
\end{split}
\end{equation}
The next step is to show the convergence of $(\int_{\R^3}f^m\,d\v)\u^m1_{\{|\u^m|\leq\lambda\}}$ in the sense of distributions.

Note that $\nabla\u^m$ is bounded in $L^2(0,T,L^2(\O))$ and \eqref{averageestimate2}, we have
\begin{equation*}
\left(\int_{\R^3}f^m\,d\v\right)\u^m1_{\{|\u^m|\leq\lambda\}}\to\left(\int_{\R^3}f\;d\v\right)\u1_{\{|\u|\leq\lambda\}}
\end{equation*}
in the sense of distributions.
Therefore,
\begin{equation*}
\left(\int_{\R^3}f^n\,d\v\right)\chi_{\lambda}(\u^n)\to\left(\int_{\R^3}f\,d\v\right) \chi_{\lambda}(\u)
\end{equation*}
in the sense of distributions.
Applying these convergence results, one concludes that $(\u,f)$ is a weak solution to the following system
\begin{equation}
\label{approximation1}
\begin{split}
&\partial_t\u+\u\cdot\nabla\u+\nabla p-\D \u=-\int_{\R^3}(\u-\v)f\,d\v1_{\{|\u|\leq\lambda\}},
\\&\Dv\u=0,
\\&\partial_tf+\v\cdot\nabla_x f+\Dv_{\v}((\chi_{\lambda}(\u)-\v)f)=0
\end{split}
\end{equation}
with its initial data $$u(0,x)=\u_0(x),\;\,f(0,x,\v)=f(0,x,\v),$$ and boundary conditions
$$\u|_{\partial\O}=0,\;\,\text{ and }\;\;\, f(t,x,\v)=f(t,x,\v^*)\;\;\text{ for }\,\,x\in\partial\O,\,\,\v\cdot\nu(x)<0.$$
Next, we show that this solution satisfies a particular energy inequality.

Because the solution $(\u^{m},f^{m})$ satisfies the energy inequality in Proposition \ref{P2}, we have
  \begin{equation*}
\begin{split}
&\int_{\O}\frac{1}{2}|\u^{m}|^2\,dx+\int_{\O\times\R^3}f^{m}(1+|\v|^2)\,d\v\, dx\\&+\int_0^T\int_{\O}|\nabla\u^{m}|^2\,dx\,dt
+\int_0^T\int_{\O\times\R^3}f^{m}|\chi_{\lambda}(\u^{m})-\v|^2\,d\v\, dx\,dt
\\&\leq\int_{\O}\frac{1}{2}|\u_0|^2\,dx+\int_{\O\times\R^3}f_0(1+|\v|^2)\,d\v \,dx.
\end{split}
\end{equation*}

The difficulty of passing the limit for the energy inequality is the convergence of the term $\int_0^T\int_{\O\times\R^3}f^{m}|\chi_{\lambda}(\u^{m})-\v|^2\,d\v\, dx\,dt$.
Here we write the term as:
\begin{equation}
\label{energyinequality}
\begin{split}
&\int_0^T\int_{\O\times\R^3}f^{m}|\chi_{\lambda}(\u^{m})-\v|^2\,d\v\, dx\,dt
\\&\quad\quad\quad\quad\quad\quad=\int_0^T\int_{\O\times\R^3}\left(f^{m}|\chi_{\lambda}(\u^{m})|^2-2f^{m}\chi_{\lambda}(\u^{m})\v+f^m|\v|^2\right)\,dx d\v dt.
\end{split}
\end{equation}
By \eqref{3.2+}, we have
\begin{equation}
\label{convergenceofcutu}
\chi_{\lambda}(\u_m)\to\chi_{\lambda}(\u)\;\;\text{ in } L^{2}(0,T;L^6(\O)).
\end{equation}
Let us look at
\begin{equation}
\begin{split}
\label{energyinequalityconvergence}
&\left|\int_0^T\int_{\O\times\R^3}f^{m}|\chi_{\lambda}(\u^{m})|^2\,d\v dx dt- \int_0^T\int_{\O\times\R^3}f|\chi_{\lambda}(\u)|^2\,d\v dxdt\right|
\\&\leq\int_0^T\int_{\O}\left(\int_{\R^3}(f^m-f)\,d\v\right)|\chi_{\lambda}(\u^m)|^2d\,dx
+\int_0^T\int_{\O}\left(\int_{\R^3}f\,d\v\right)\left(|\chi_{\lambda}(\u^m)|^2-|\chi_{\lambda}(\u)|^2\right)\,dx.
\end{split}
\end{equation}
Applying \eqref{averageestimate2} and \eqref{convergenceofcutu} to \eqref{energyinequalityconvergence}, we deduce that
\begin{equation*}
\int_0^T\int_{\O\times\R^3}f^{m}|\chi_{\lambda}(\u^{m})|^2\,d\v dx dt\to \int_0^T\int_{\O\times\R^3}f|\chi_{\lambda}(\u)|^2\,d\v dxdt
\end{equation*}
as $m\to\infty.$
Similarly,
\begin{equation*}
\int_0^T\int_{\O\times\R^3}\v f^{m}\chi_{\lambda}(\u^{m})\,d\v dx dt\to
\int_0^T\int_{\O\times\R^3}\v f\chi_{\lambda}(\u)\,d\v dxdt
\end{equation*}
for all $t>0$.

Finally, because
 $$f^{m}\rightharpoonup f\quad\quad\text{ weak star in } L^{\infty}(0,T;L^{p}(\O\times\R^3))$$
 for all $p\in(1,\infty]$ and $m_2f^m$ is bounded in $L^{\infty}(0,T;L^1(\O))$, then for any fixed $R>0,$ we have
  \begin{equation*}
  \int_0^T\int_{\O\times\R^3}f^m|\v|^2\,dx d\v dt=\int_0^T\int_{\O\times\R^3}\chi(|\v|<R)|\v|^2f^m\,dx d\v dt+O(\frac{1}{R})
  \end{equation*}
  uniformly in $m$, where $\chi$ is the characteristic function of the ball of $\R^3$ of radius $R$.

   Letting $m\to\infty,$ then $R\to\infty$, we find \begin{equation*}
  \int_0^T\int_{\O\times\R^3}f^m|\v|^2\,dx d\v dt\to   \int_0^T\int_{\O\times\R^3}f|\v|^2\,dx d\v dt.
  \end{equation*}
  Thus, we have proved
  \begin{equation}
  \label{energyinequalityconvergence2}
  \int_0^T\int_{\O\times\R^3}f^{m}|\chi_{\lambda}(\u^{m})-\v|^2\,d\v\, dx\,dt\to \int_0^T\int_{\O\times\R^3}f|\chi_{\lambda}(\u)-\v|^2\,d\v\, dx\,dt
  \end{equation}
  as $m\to \infty.$

Letting $m$ go to infinity,  using the convexity of the energy, the weak convergence of $f^{m}$ and $\u^{m}$, and \eqref{energyinequalityconvergence2}, we deduce
  \begin{equation*}
\begin{split}
&\int_{\O}\frac{1}{2}|\u|^2\,dx+\int_{\O\times\R^3}f(1+|\v|^2)\,d\v \,dx\\&+\int_0^T\int_{\O}|\nabla\u|^2\,dx\,dt
+\int_0^T\int_{\O\times\R^3}f|\chi_{\lambda}(\u)-\v|^2d\v \,dx\,dt
\\&\leq\int_{\O}\frac{1}{2}|\u_0|^2\,dx+\int_{\O\times\R^3}f_0(1+|\v|^2)\,d\v\, dx.
\end{split}
\end{equation*}
Thus, we have proved the following result:
\begin{Proposition}
\label{P3}
For any $T>0$, there is a weak solution $(\u^{\lambda},f^{\lambda})$ to \eqref{approximation1}
with the initial data $$\u^{\lambda}(0,x)=\u_0(x),\,\,f^{\lambda}(0,x,\v)=f_0(x,\v),$$
and boundary condition $$\u^{\lambda}|_{\partial\O}=0,\,\,\text{and}\,\,\,f^{\lambda}(t,x,\v)=f^{\lambda}(t,x,\v^*)\,\,\,\,\text{ for } \;x\in\partial\O,\,\v\cdot\nu(x)<0.$$
In additional, the solution satisfies the following energy inequality:
 \begin{equation*}
\begin{split}
&\int_{\R^3}\frac{1}{2}|\u^{\lambda}|^2\,dx+\int_{\O\times\R^3}f^{\lambda}(1+|\v|^2)d\v dx\\&+\int_0^T\int_{\O}|\nabla\u^{\lambda}|^2\,dx\,dt
+\int_0^T\int_{\O\times\R^3}f^{\lambda}|\chi_{\lambda}(\u^{\lambda})-\v|^2\,d\v\, dx\,dt
\\&\leq\int_{\O}\frac{1}{2}|\u_0|^2\,dx+\int_{\O\times\R^3}f_0(1+|\v|^2)\,d\v\, dx.
\end{split}
\end{equation*}
\end{Proposition}

\subsection{Passing the limit as $\lambda\to\infty$}
The last step of showing the global weak solution is to pass the limit as $\lambda$ goes to infinity. First, we let $(\u^{\lambda},f^{\lambda})$ be a solution constructed by Proposition \ref{P3}. It is easy to find that all estimates for $(\u^m,f^m)$ still hold for $(\u^{\lambda},f^{\lambda})$.
So we can treat these terms as before.

It only remains to show that we can pass the limit in the coupling terms $\chi_{\lambda}(\u^{\lambda})f^{\lambda}$ and $\int_{\R^3} f^{\lambda}\,d\v 1_{\{|\u^{\lambda}|\leq\lambda}\}=\int_{\R^3} f^{\lambda} \,d\v\chi_{\lambda}(\u^{\lambda})$.
 Here, we treat these terms
 as follows
 \begin{equation}
 \label{a1}
 \int_{\R^3}f^{\lambda}\u^{\lambda}\,d\v1_{\{|\u^{\lambda}|\leq\lambda\}}=
 \int_{\R^3}f^{\lambda}\u^{\lambda}\,d\v-
 \int_{\R^3}f^{\lambda}\u^{\lambda}\,d\v1_{\{|\u^{\lambda}|>\lambda\}},
 \end{equation}
 and for the second term in \eqref{a1},
\begin{equation}
\begin{split}
\label{a2}
& \|\int_{\R^3}f^{\lambda}\u^{\lambda}\,d\v1_{\{|\u^{\lambda}|>\lambda\}}\|_{L^{1}(0,T;\O)}
\\ & \leq \|
 \int_{\R^3}f^{\lambda}\,d\v\|_{L^{\infty}(0,T;L^2(\O)}\|\u^{\lambda}\|_{L^{2}(0,T;L^{6}(\O))}
 \|1_{\{|\u^{\lambda}|>\lambda\}}\|_{L^{2}(0,T;L^{6}(\O))}
\\&\leq\|\int_{\R^3}f^{\lambda}d\v\|_{L^{\infty}(0,T;L^{2}(\O)}\|\u^{\lambda}\|_{L^{2}(0,T;L^{6}(\O))}(\frac{\|\u^{\lambda}\|_{L^{2}(0,T;L^{6}(\O))}}{\lambda})
 \\& \leq\frac{C}{\lambda}\to 0
 \end{split}
 \end{equation}
 as $\lambda\to\infty,$
  where we used Sobolev embedding theorem.

   On the other hand, we have
\begin{equation*}
\partial_t(\int_{\R^3}f^{\lambda}\,d\v)+\Dv_{x}(\int_{\R^3}\v f^{\lambda}\,d\v)=0,
\end{equation*}
which implies that
$ \partial_t(\int_{\R^3}f^{\lambda}\,d\v)$ is bounded in $L^2(0,T;H^{-1})$.
This, with the help of $\nabla\u^{\lambda}$ bounded in $L^2((0,T)\times\O)$, yields that
\begin{equation}
\label{a3}
\u^{\lambda}\left(\int_{\R^3}f^{\lambda}\,d\v\right)\to\u\left(\int_{\R^3}f\,d\v\right)\;\;\text{as } \lambda\to\infty
\end{equation}
in the sense of distributions.

By \eqref{a1}-\eqref{a3}, one deduces that
\begin{equation*}
 \u^{\lambda}\left(\int_{\R^3}f^{\lambda}\,d\v1_{\{|\u^{\lambda}|\leq\lambda\}}\right)\to \u\left(\int_{\R^3}f\, d\v\right)\;\;\text{as } \lambda\to\infty
\end{equation*}
  in the sense of distributions.
Thus
we can pass the limit in the weak solutions of \eqref{1.1} as $\lambda\to\infty$.
We remark that the solution $(\u^{\lambda},f^{\lambda})$ satisfies the energy inequality in Proposition \ref{P3}:
  \begin{equation*}
\begin{split}
&\int_{\O}\frac{1}{2}|\u^{\lambda}|^2\,dx+\int_{\O\times\R^3}f^{\lambda}(1+|\v|^2)\,d\v\, dx\\&+\int_0^T\int_{\O}|\nabla\u^{\lambda}|^2\,dx\,dt
+\int_0^T\int_{\O\times\R^3}f^{\lambda}|\chi_{\lambda}(\u^{\lambda})-\v|^2d\v \,dx\,dt
\\&\leq\int_{\O}\frac{1}{2}|\u_0|^2\,dx+\int_{\O\times\R^3}f_0(1+|\v|^2)\,d\v\, dx.
\end{split}
\end{equation*}

Using the same approach as in last subsection, letting $\lambda$ go to infinity,  using the convexity of the energy and the weak convergence of $f^{\lambda}$ and $\u^{\lambda}$, we deduce
  \begin{equation*}
\begin{split}
&\int_{\O}\frac{1}{2}|\u|^2\,dx+\int_{\O\times\R^3}f(1+|\v|^2)\,d\v\, dx\\&+\int_0^T\int_{\O}|\nabla\u|^2\,dx\,dt
+\int_0^T\int_{\O\times\R^3}f|\u-\v|^2\,d\v \,dx\,dt
\\&\leq\int_{\O}\frac{1}{2}|\u_0|^2\,dx+\int_{\O\times\R^3}f_0(1+|\v|^2)\,d\v\, dx.
\end{split}
\end{equation*}
So we have proved Theorem \ref{T1}.

\section{Uniqueness in The Two Dimensional Space}
The goal of this section is to establish the uniqueness of global solutions in the two dimensional space. For that purpose, we shall study the regularity first.
\subsection{Regularity}
The existence of global weak solution to \eqref{1.1} was obtained by Theorem \ref{T1}.
We multiply the first equation of \eqref{1.1} by $\u_{t}$ and use integration by parts over $\O$ to obtain,
\begin{equation*}
\begin{split}
&\int_{\O}|\u_{t}|^2dx+\partial_{t}\int_{\O}|\nabla\u|^2\,dx \\&\leq\int_{\O}|m_{0}f\cdot\u\cdot\u_{t}|\,dx+\int_{\O}|\u\cdot\nabla\u\cdot\u_{t}|\,dx
+\int_{\O}|m_{1}f||\u_{t}|\,dx.
\end{split}
\end{equation*}
By the Cauchy-Schwarz inequality, we deduce
\begin{equation}
\label{5.2}
\begin{split}
& \|\frac{\partial\u}{\partial t}\|_{L^{2}(\O)}^2+\frac{d}{dt}\|\nabla\u \|_{L^{2}(\O)}^2
\\& \leq \|m_{0}f\|_{L^{4}(\O)}\|\u\|_{L^{4}(\O)}\|\u_{t}\|_{L^{2}(\O)}+
\|\u_{t}\|_{L^{2}(\O)}\|\u\cdot\nabla\u\|_{L^{2}(\O)}
\\&+\|m_{1}f\|_{L^{2}(\O)}\|\u_{t}\|_{L^{2}(\O)}.
\end{split}
\end{equation}
By Theorem \ref{T1}, we have $$\u \in L^{\infty}(0,T;L^2(\O)), \nabla\u\in L^{2}(0,T;L^{2}(\O)).$$
Using the Gagliardo-Nirenberg inequality
\begin{equation}
\label{5.3}
\|v\|_{L^{4}}\leq C\|v\|_{L^2}^{1/2}\|\nabla v\|_{L^2}^{1/2},
\end{equation}
one obtains that
\begin{equation}
\label{5.4}
\int_{0}^{T}\int_{\O}|\u|^4\,dx\,dt\leq C\|\u\|_{L^{\infty}(0,T;L^{2}(\O))}^2\|\nabla\u\|_{L^{2}(0,T;L^{2}(\O))}^2\leq C.
\end{equation}
Since $\u \in L^{2}(0,T;H^1_0(\O))$, by the Sobolev imbedding inequality, we obtain $\u \in L^{2}(0,T;L^{p}(\O))$ for any $1\leq p< \infty.$

Thanks to Lemma \ref{L2} with $d=2$, we have
\begin{equation*}
M_{6}f<\infty \;\;\text{ for any } 0\leq t\leq T,
\end{equation*}
if $m_{6}f_{0}<\infty.$

Let us estimate $m_{0}f$ in the two dimensional space:
\begin{equation*}
\begin{split}
m_{0}f&=\int_{\R^2}f\,d\v=\int_{|\v|<r}f\, d\v+\int_{|\v|\geq r}f\,d\v
\\& \leq C\|f\|_{L^{\infty}}r^{2}+\frac{1}{r^k}\int_{|\v|\geq r}|\v|^k f \,d\v
\end{split}
\end{equation*}
for all $k\geq 0.$
Letting $r=(\int_{\R^2}|\v|^k f d\v)^{\frac{1}{k+2}}$, then
\begin{equation*}
m_{0}f\leq C(\|f\|_{L^{\infty}}+1)(\int_{\R^2}|\v|^k f \,d\v)^{\frac{2}{k+2}}
\end{equation*}for all $k\geq0.$
Taking $k=6,$ then $m_{0}f\leq C(m_{6}f)^{1/4},$ which means
\begin{equation}
\label{5.5}
\|m_{0}f\|_{L^{\infty}(0,T;L^{4}(\O))}< \infty.
\end{equation}
Similarly, we have
\begin{equation}
\label{5.6}
\|m_{1}f\|_{L^{\infty}(0,T;L^{2}(\O))}< \infty.
\end{equation}
By \eqref{5.2}-\eqref{5.6}, we have for all $\varepsilon>0$
\begin{equation}
\label{5.7}
\|\frac{\partial\u}{\partial t}\|_{L^{2}(\O)}^2+\frac{d}{dt}\|\nabla\u\|_{L^2(\O)}^2\leq \frac{C(t)}{\varepsilon}(1+\|\nabla\u\|_{L^{2}(\O)}^2)+\varepsilon\|D^2\u\|_{L^{2}(\O)}^2
\end{equation}
where $C(t)\geq 0$, and $\int_{0}^{T}C(t)dt\leq C$ for all $T>0.$

Next, observe that for all $t\geq 0$ in view of \eqref{1.1}, we have
\begin{equation*}
\begin{split}
&\quad\quad\quad\|-\D\u+\nabla p\|_{L^{2}(\O)}
\\& \leq C\left(\|m_{1}f\|_{L^2(\O)}+\|\frac{\partial \u}{\partial t}\|_{L^2(\O)}+\||\u||\nabla\u|\|_{L^2(\O)}+\|m_{0}f\cdot\u\|_{L^2(\O)}    \right).
\end{split}
\end{equation*}
Due to $\Dv \u=0,$ by the classical regularity on Stokes equations, we obtain
\begin{equation*}
\begin{split}
\quad\quad\quad\quad &\|\u\|_{H^2(\O)}
\\ &\leq  C\left(\|m_{1}f\|_{L^2(\O)}+\|\frac{\partial \u}{\partial t}\|_{L^2(\O)}+\||\u||\nabla\u|\|_{L^2(\O)}+\|m_{0}f\cdot\u\|_{L^2(\O)}    \right).
\end{split}
\end{equation*}
Following the same argument of \eqref{5.7}, we have for all $\varepsilon'>0$,
\begin{equation}
\label{5.8}
\|\u\|_{H^2(\O)}
\leq  \frac{1}{\varepsilon'}C_{1}(t)+C\|\frac{\partial \u}{\partial t}\|_{L^2{(\O)}}+\varepsilon'\|\u\|_{H^{2}(\O)}
\end{equation}
where $C_{1}(t)\geq 0$, and $\int_{0}^{T}C_{1}^2(t)dt\leq C.$
Choosing $\varepsilon'=\frac{1}{2},$ we obtain
\begin{equation}
\label{5.9}
\|\u\|_{H^{2}(\O)}^2\leq C_{2}(t)+C\|\frac{\partial \u}{\partial t}\|_{L^2(\O)}
\end{equation}
where $C_{2}(t)\geq 0$, and $\int_{0}^{T}C_{2}^2(t)dt\leq C.$
Inserting \eqref{5.9} in \eqref{5.7} and choose $\varepsilon=1/2C,$ we obtain for all $t\geq0$
\begin{equation}
\label{uniqueness1+}
\|\frac{\partial \u}{\partial t}\|_{L^2(\O)}^2
+\frac{d}{d t}\|\nabla\u\|_{L^{2}(\O)}^2\leq C_{3}(t)(1+\|\nabla\u\|_{L^{2}(\O)}^2)
\end{equation}
where $C_{3}(t)\geq 0$, and $\int_{0}^{T}C_{3}^2(t)dt\leq C.$

Applying Gronwall's inequality to  \eqref{uniqueness1+}, we obtain
\begin{equation*} \frac{\partial \u}{\partial t} \quad \text{ is bounded in }L^{2}(\O\times(0,T)),
\end{equation*}
and
\begin{equation*}
\u \quad\text{is bounded in }L^{\infty}(0,T;H^{1}_0(\O)).
\end{equation*}
This, with the help \eqref{5.9}, implies that
\begin{equation*}
\u \;\;\text{ is bounded in } L^2(0,T;H^2(\O)).
\end{equation*}
Here, we need to rely on the following Lemma which is a very special case of interpolation theorem of Lions-Magenes. We refer the readers to \cite{T} for the proof of this lemma.
\begin{Lemma}
\label{L5.1}
Let $V\subset H\subset V'$ be three Hilbert spaces, $V'$ is a dual space of $V$. If a function $\u$ belong to $L^{2}(0,T;V)$ and its derivative $\u'$ belongs to $L^2(0,T;V')$ then $\u$ is almost everywhere equal to a function continuous from $[0,T]$ into $H$.
\end{Lemma}
Thanks to
\begin{equation*}
\frac{\partial \u}{\partial t} \in L^{2}(\O\times(0,T)),\text{ and } \u \in L^{\infty}(0,T;H^{1}(\O) \cap L^2(0,T;H^2(\O)),
\end{equation*}
we conclude that $\u\in C([0,T],H^1(\O))$ by Lemma \ref{L5.1}, consequently $f\in C^1([0,T];L^{\infty}(\R^2\times\O)).$
\subsection{Uniqueness of solutions}

To show the uniqueness, we rely on the following parabolic regularity due to \cite{E,SA,T}:

\begin{Lemma}
\label{L4.2}
If $\u$ solves
\begin{equation*}
\begin{split}
&\partial_t\u-\D\u+\nabla p=F,
\\& \u(t=0)=\u_0,\;\;\;\u|_{\partial\O}=0,
\end{split}
\end{equation*}
on some time interval $(0,T),$ then we have
\begin{equation*}
\|\u\|_{L^{\infty}([0,T);H^1_0)\cap L^2(0,T;H^2)}\leq C\left(\|F\|_{L^2((0,T)\times\O)}+\|\u_0\|_{H^1_0}\right).
\end{equation*}
\end{Lemma}

Now we are ready to show the uniqueness. Let $(\u_1,f_1)$ and $(\u_2,f_2)$ be two different solutions to \eqref{1.1}-\eqref{boundary condition}. Let $\bar{\u}=\u_1-\u_2$, and $\bar{f}=f_1-f_2$, then we have the following equations:
\begin{equation}
\label{5.10}
\begin{split}
&\bar{\u}_t+\nabla p-\D \bar{\u}=-\int_{\R^2}(\bar{\u} f_{1}+\u_{2}\bar{f}-\v \bar{f})\,d\v-(\bar{\u}\cdot\nabla\u_1+\u_2\cdot\nabla\bar{\u})
\\&\Dv\bar{\u}=0,
\\&\bar{f}_t+\v\cdot\nabla_x\bar{f}+\Dv_{\v}(\bar{\u} f_1+\u_2\bar{f}-\v \bar{f})=0
\end{split}
\end{equation}
in $\O\times\R^2\times(0,T),$ subject to the following initial data
\begin{equation*}
\bar{\u}(0,x)=0,\quad\quad \bar{f}(0,x,\v)=0,
\end{equation*}
and boundary condition
\begin{equation*}
\bar{\u}|_{\partial\O}=0,\;\;\;\bar{f}(t,x,\v^*)=\bar{f}(t,x,\v)\;\;\text{ for }x\in\partial\O, \;\v\cdot\nu(x)<0.
\end{equation*}

Here, we denote the space $X= L^{\infty}(0,T;H^1_0)\cap L^{2}(0,T;H^2_0).$
Applying Lemma \ref{L4.2} with $\u_0=0$, we have the following regularity:
\begin{equation}
\label{5.11}
\begin{split}
&\|\bar{\u}\|_{X}\leq C\|\int_{\R^2}(\bar{\u} f_{1}+\u_{2}\bar{f}-\v \bar{f})d\v+(\bar{\u}\cdot\nabla\u_1+\u_2\cdot\nabla\bar{\u})\|_{L^2((0,T)\times\O)}
\\&=J_1+J_2.
\end{split}
\end{equation}

For $J_1$:
\begin{equation}
\label{5.12}
\begin{split}
&\|\int_{\R^2}(\bar{\u} f_{1}+\u_{2}\bar{f}-\v \bar{f})d\v\|_{L^{2}((0,T)\times\O)}
\\&=\|\bar{\u} m_0f_1+\u_2 m_0\bar{f}-m_1\bar{f}\|_{L^2((0,T)\times\O)}
\\ & \leq \|\bar{\u}\|_{L^{\infty}(0,T;L^{p_1}(\O))}\|m_0f_1\|_{L^{2}(0,T;L^{3}(\O))}+ \|1\|_{L^6(0,T;L^x(\O))}\|\u_2\|_{L^{\infty}(0,T;L^{p_{2}}(\O))}\|m_0\bar{f}\|_{L^{3}((0,T)\times\O)}
\\& +\|m_1\bar{f}\|_{L^{2}((0,T)\times\O)}
\\& \leq \sup_{t\in[0,T]}\|m_0f_1\|_{L^{3}(\O)}T^{\frac{1}{2}}\|\bar{\u}\|_{X}+
CT^{\frac{1}{6}}\|\u_2\|_{L^{\infty}(0,T;L^{p}(\O))}\|m_0\bar{f}\|_{L^{\infty}(0,T;L^{3}(\O))}
\\&+\|m_1\bar{f}\|_{L^{2}((0,T)\times\O)},
\end{split}
\end{equation}
where $p_1=6$, $x=\frac{6p_{2}}{p_{2}-6},$ for any $p_{2}>6,$ $C$ depends on the domain $\O$.
And for $J_2$:
\begin{equation}
\label{5.13}
\begin{split}
&\|\bar{\u}\cdot\nabla\u_1+\u_2\cdot\nabla\bar{\u}\|_{L^2((0,T)\times\O)}
\\ & \leq\|1\|_{L^{4}(0,T;L^y(\O))}\|\bar{\u}\|_{L^{\infty}(0,T;L^p(\O))}\|\nabla\u_1\|_{L^4((0,T)\times\O)}
\\&+\|1\|_{L^4(0,T;L^{\infty}(\O))}\|\u_2\|_{L^{4}(0,T;L^{\infty}(\O))}\|\nabla\bar{\u}\|_{L^{\infty}(0,T;L^2(\O))}
\\&\leq CT^{\frac{1}{4}}\|\nabla\u_1\|_{L^4((0,T)\times\O)}\|\u\|_{X}+
CT^{\frac{1}{4}}\|\u_2\|_{L^{4}(0,T;L^{\infty}(\O))}\|\u\|_{X},
\end{split}
\end{equation}
where $y=\frac{4p}{p-4}$ for any $p>4$, $C$ only depends on the domain $\O$, and we used the Gagliardo-Nirenberg inequality for $\nabla\u_1$, the Gagliardo-Nirenberg inequality and embedding inequality for $\u_2$.
By \eqref{5.11}-\eqref{5.13}, we can choose small $T$ such that
\begin{equation}
\label{chooseT}
\begin{split}&\sup_{t\in[0,T]}\|m_0f_1\|_{L^{3}(\O)}T^{\frac{1}{2}}+C\|\u_2\|_{L^{\infty}(0,T;L^{p}(\O))}T^{\frac{1}{6}}
\\&+CT^{\frac{1}{4}}\|\nabla\u_1\|_{L^4((0,T)\times\O)}+CT^{\frac{1}{4}}\|\u_2\|_{L^{4}(0,T;L^{\infty}(\O))}\leq \frac{1}{2},
\end{split}
\end{equation}
then
\begin{equation}
\label{5.14}
\|\bar{\u}\|_{X}\leq \|m_0\bar{f}\|_{L^{\infty}(0,T;L^3(\O))}+\|m_1\bar{f}\|_{L^{2}((0,T)\times\O)}.
\end{equation}
The next step is to show $\|m_0\bar{f}\|_{L^{3}((0,T)\times\O)}+\|m_1\bar{f}\|_{L^{2}((0,T)\times\O)}$ can be controlled by $\|\u\|_{X}$.

In the two dimensional space, we have $$\int_0^T\|m_{0}\bar{f}\|_{L^3(\O)}^3\,dt\leq C\int_0^T\int_{\O}\int_{\R^2}|\v|^4\bar{f} \,d\v \,dx\,dt,$$
and $$\int_0^T\|m_{1}\bar{f}\|_{L^2(\O)}^2\,dt\leq C\int_0^T\int_{\O}\int_{\R^2}|\v|^4 \bar{f}\, d\v\, dx\,dt.$$
Pluging them into \eqref{5.14}, and choosing $T$ small enough again, one deduces that
\begin{equation}
\label{5.15}
\|\bar{\u}\|_X\leq C\int_0^T\int_{\O}\int_{\R^2}|\v|^4 \bar{f} d\v\,dx\,dt.
\end{equation}

We multiply the second equation of \eqref{5.10} by $|\v|^k$ for $k\geq1,$ and use integration by parts over $\O\times\R^2$:
\begin{equation}
\label{5.16}
\begin{split}
&\partial_{t}\int_{\O}\int_{\R^2} \bar{f}|\v|^k\,d\v\, dx+k\int_{\O}\int_{\R^2} \bar{f}|\v|^k\, d\v\, dx
\\&\quad\quad\quad\quad\quad\quad=k \int_{\O}\int_{\R^2}\bar{\u} f_{1}|\v|^{k-1}\,d\v\, dx+k\int_{\O}\int_{\R^2}\u_{2}\bar{f} |\v|^{k-1}\,d\v\, dx.
\end{split}
\end{equation}
We estimate the right hand side terms of \eqref{5.16}:
\begin{equation}
\label{5.17}
\begin{split}
&k \int_{\O}\int_{\R^2}\bar{\u} f_{1}|\v|^{k-1}d\v dx+k\int_{\O}\int_{\R^2}\u_{2}\bar{f} |\v|^{k-1}\,d\v\, dx
\\ & \leq C \int_{\O}|\bar{\u}||m_{k-1}f_1|\,dx+C\int_{\O}|\u_{2}||m_{k-1}\bar{f}|\,dx
\\ & \leq C\|\bar{\u}\|_{L^2(\O)}\|m_{k-1}f_1\|_{L^2(\O)}+C\|\u_{2}\|_{H^2_0(\O)}\|m_{k-1}\bar{f}\|_{L^1(\O)}.
\end{split}
\end{equation}
By \eqref{5.16} and \eqref{5.17}, we have
\begin{equation}
\begin{split}
\label{5.18}
&\sup_{t\in(0,T)}\int_{\O}\int_{\R^2} \bar{f}|\v|^k\,d\v\, dx+k\int_{0}^T\int_{\O}\int_{\R^2} \bar{f}|\v|^k\, d\v\, dx\, dt
\\ & \leq \int_0^T\|\bar{\u}\|_{L^2(\O)}\|m_{k-1}f_1\|_{L^2(\O)}dt+\int_0^T\|\u_2\|_{H^2_0(\O)}\|m_{k-1}\bar{f}\|_{L^1(\O)}dt
\\& \leq \|\bar{\u}\|_{L^2((0,T)\times\O)}\|m_{k-1}f_1\|_{L^2(\O)}T^{\frac{1}{2}}+
\|\u_2\|_{L^2(0,T;H^2_0(\O))}T^{\frac{1}{2}}\sup_{t\in[0,T]}\int_{\O}\int_{\R^2}|\v|^{k-1}\bar{f}\,d\v\,dx.
\end{split}
\end{equation}
for all $k\geq1.$
Meanwhile, we integrate the third equation in \eqref{5.10} over $\O\times\R^2$ and use integration by parts:
\begin{equation}
\label{5.19}
\int_{\O}\int_{\R^2} \bar{f}\,d\v\, dx=\int_{\O}\int_{\R^2} \bar{f}_{0}\,d\v \,dx=0.
\end{equation}
Using \eqref{5.18}-\eqref{5.19} and by induction, we deduce
\begin{equation}
\label{5.20}
\begin{split}
&\int_0^T\int_{\O}\int_{\R^2} \bar{f}|\v|^{4}\,d\v \,dx\,dt
\leq (\sup_{t\in[0,T]}\|m_3f_1\|_{L^2(\O)}T^{\frac{1}{2}}+\|\u_2\|_{X}\sup_{t\in[0,T]}\|m_2f_1\|_{L^2(\O)}T
\\&+\|\u_2\|_{X}^2\sup_{t\in[0,T]}\|m_1f_1\|_{L^2(\O)}T^{\frac{3}{2}}
+\|\u_2\|_{X}^3\sup_{t\in[0,T]}\|m_0f_1\|_{L^2(\O)}T^2)\|\bar{\u}\|_{X}.
\end{split}
\end{equation}
Thanks to \eqref{5.20} and \eqref{5.15}, choosing $T>0$ small enough, we obtain:
\begin{equation*}
\|\bar{\u}\|_{X}\leq \frac{1}{2}\|\bar{\u}\|_{X}.
\end{equation*}
Thus, $\bar{\u}=0$, and hence $\u_1=\u_2$ on a small time interval. Thus, we have proved the uniqueness of $\u$ on a small time interval.

On the other hand, we have the following equation on the same time interval,
\begin{equation}
\label{5.21}
\bar{f}_{t}+\v\nabla \bar{f}+\Dv_{\v}((\u_2-\v)\bar{f})=0\;\; \text{ in } \O\times\R^2\times(0,T],
\end{equation}
with its initial data $$\bar{f}_{0}=0,$$
and boundary condition $$\bar{f}(t,x,\v)=\bar{f}(t,x,\v^*)\;\;\text{ for } \,x\in\partial\O,\,\v\cdot\nu(x)<0.$$
By \eqref{5.21}, we have
\begin{equation*}
\|\bar{f}\|_{L^{\infty}((0,T]\times\O\times\R^2)}\leq C\|\bar{f}_{0}\|_{L^{\infty}((0,T]\times\O\times\R^2)},
\end{equation*}
which yields $f_1=f_2.$
So we have proved the uniqueness of solution $(\u,f)$ on a small time interval $[0,T_0].$ For any given $T>0$, we consider the maximal interval of the uniqueness, $T_1=\sup T_0\leq T,$ such that the solution is unique on $[0,T_0].$ The main goal is to prove that $T_1$ can be taken to be equal to $\infty.$ For any given $T_0>0,$ we use
\begin{equation*}
\bar{\u}(T_0,x)=0,\;\;\bar{f}(T_0,x,\v)=0,
\end{equation*}as the new initial data to \eqref{5.10}.
Applying the same argument to equation \eqref{5.10} with the new data, the uniqueness of solution can be extended to $[0,T_0+T^{*}]$ for a small number $T^{*}>0$.

By \eqref{chooseT} and \eqref{5.20}, $T^*$ can be chosen only depending on the upper bounds of $$\sup_{t\in[0,T]}\|m_0f_1\|_{L^{3}(\O)},\;\;\|\u_1\|_{X},\;\;\|\u_2\|_{X},\;\;\sup_{t\in[0,T]}\|m_if_1\|_{L^2(\O)}\;\text{ for } i=0,1,2,3.$$
 By the regularity of $\u$ in Section 4.1, $\u_1,$ $\u_2$ are uniformly bounded in the space $X$. Applying the same argument of \eqref{5.5} and Lemma \ref{L2}, we can show that the other terms are uniformly bounded for all time $t\geq 0$.
All such terms are uniformly bounded for all time $t\geq0$.
Thus, $T^*$ can be chosen not depending on the initial data at time $T_0$. In fact, we can choose $T^*=T_0$.
One can then repeat the argument many times and obtain the uniqueness of $(\u,f)$ on the whole time.

\bigskip


\section*{Acknowledgments}

The author thanks Professor Marshall Slemrod and Professor Dehua Wang for valuable discussions
and suggestions. He is also indebted to the anonymous referees for improving this article through their comments.

\end{document}